\documentclass[11pt]{amsart}
\usepackage{amssymb}
\usepackage{hyperref}
\theoremstyle{plain}
\newtheorem{theorem}{Theorem}
\newtheorem{corollary}{Corollary}

\newtheorem{proposition}{Proposition}

\theoremstyle{definition}

\theoremstyle{remark}

\numberwithin{equation}{section}

\setbox0=\hbox{$+$}
\newdimen\plusheight
\plusheight=\ht0
\def\+{\;\lower\plusheight\hbox{$+$}\;}

\setbox0=\hbox{$-$}
\newdimen\minusheight
\minusheight=\ht0
\def\-{\;\lower\minusheight\hbox{$-$}\;}

\setbox0=\hbox{$\cdots$}
\newdimen\cdotsheight
\cdotsheight=\plusheight
\def\cds{\lower\cdotsheight\hbox{$\cdots$}}

\begin{document}
\title[Powers of a matrix and combinatorial identities ]
 { Powers of a matrix and combinatorial identities}
 \author{J. Mc Laughlin}
\address{Mathematics Department\\
 Trinity College\\
300 Summit Street, Hartford, CT 06106-3100}
\email{james.mclaughlin@trincoll.edu}
\author{ B. Sury}
\address{Stat-Math Unit,\\
Indian Statistical Institute,\\
8th Mile Mysore Road,
Bangalore 560 059,
India.}
\email{sury@isibang.ac.in}

\date{May 4th, 2004}

\keywords{Polynomial and Combinatorial Identities}
\subjclass{05A19, 15A99}

\begin{abstract}

In this article we obtain a general polynomial identity in $k$ variables, where
$k\geq 2$ is an arbitrary positive integer.
We use this identity to give a closed-form expression for the entries of
the powers of a
$k \times k$ matrix.
Finally, we use these results to derive various combinatorial identities.
\end{abstract}

\maketitle

\section{Introduction}

In \cite{S93}, the second author had observed that the
following `curious' polynomial identity holds:
\[
\sum (-1)^i \binom{n-i }{ i} (x+y)^{n-2i}(xy)^i
= x^n + x^{n-1}y + \cdots + xy^{n-1} + y^n.
\]
The proof was simply observing that both sides satisfied the same recursion.
He had also observed (but not published the result) that this
recursion defines in a closed form the entries of the powers of a
$2 \times 2$ matrix in terms of its trace and determinant and the entries of
the original matrix. The first author had independently discovered this fact
and derived several combinatorial identities as consequences
\cite{McL04}.

In this article, for a general $k$, we obtain a polynomial identity
and show how it gives a closed-form expression for the entries of the powers of a
$k \times k$ matrix. From these, we derive some combinatorial identities as
consequences.

\section{Main Results}
{\it Throughout the paper, let $K$ be any fixed field of characteristic zero.
We also fix a positive integer $k$.} The main results are the following
two theorems:

\begin{theorem}\label{t1}
Let $x_1, \cdots, x_k$ be independent variables and let $s_1, \cdots, s_k$
denote the various symmetric polynomials in the $x_i$'s of degrees
$1,2 \cdots, k$ respectively.
Then, in the polynomial ring $K[x_1, \cdots, x_k]$, for each positive integer $n$,
one has the identity
\begin{multline*}
\sum_{r_1 + \cdots + r_k = n} x_1^{r_1}x_2^{r_2} \cdots x_k^{r_k}
=\\
\sum_{2i_2+3i_3+ \cdots +ki_k \leq n}
%\negthickspace \negthickspace \negthickspace \negthickspace \negthickspace
%\negthickspace \negthickspace \negthickspace
 c(i_2,\cdots,i_k,n)
s_1^{n-2i_2-3i_3- \cdots - ki_k}\\
\times
(-s_2)^{i_2}s_3^{i_3} \cdots ((-1)^{k-1}s_k)^{i_k},
\end{multline*}
 where
\[
c(i_2,\cdots,i_k,n) =
\frac{ (n-i_2-2i_3- \cdots -(k-1)i_k )!}{
i_2! \cdots i_k! (n-2i_2-3i_3- \cdots -(ki_k )!}.
\]
\end{theorem}

\begin{theorem}\label{t2}
Suppose $A \in M_k(K)$ and let
\[
T^k-s_1T^{k-1}+s_2T^{k-2}+ \cdots +(-1)^ks_k\,I
\]
denote its characteristic polynomial.
Then, for all $n \geq k$, one has
\[
A^n = b_{k-1}A^{k-1}+b_{k-2}A^{k-2}+ \cdots + b_0 \,I,
\]
where
{\allowdisplaybreaks
\begin{align*}
b_{k-1}& = a(n-k+1),\\
b_{k-2} &= a(n-k+2)-s_1a(n-k+1),\\
& \phantom{a}\vdots \\
b_1& = a(n-1)-s_1a(n-2)+ \cdots + (-1)^{k-2}s_{k-2}a(n-k+1),\\
b_0 &= a(n)-s_1a(n-1)+ \cdots + (-1)^{k-1}s_{k-1}a(n-k+1)\\
& = (-1)^{k-1}s_ka(n-k).
\end{align*}
}
and
\[
a(n) = c(i_2,\cdots,i_k,n)
s_1^{n-i_2-2i_3- \cdots -(k-1)i_k}
(-s_2)^{i_2}s_3^{i_3} \cdots ((-1)^{k-1}s_k)^{i_k},
\]
with
\[
c(i_2,\cdots,i_k,n) =
\frac{ (n-i_2-2i_3- \cdots -(k-1)i_k )!}{
i_2! \cdots i_k! (n-2i_2-3i_3- \cdots -(ki_k )!}.
\]
as in Theorem 1.
\end{theorem}

%\noindent
\emph{Proof of Theorems \ref{t1} and \ref{t2}.}
In Theorem 1,
if $a(n)$ denotes either side, it is straightforward to verify that
{\allowdisplaybreaks
\[
a(n) = s_1a(n-1)-s_2a(n-2)+ \cdots +(-1)^{k-1}s_ka(n-k).
\]
}
Theorem 2 is a consequence of Theorem 1 on using induction on $n$.
\begin{flushright}
$\Box$
\end{flushright}
The special cases $k=2$ and $k=3$ are worth noting for it is easier to derive
various combinatorial identities from them.

\begin{corollary}\label{c1}
(i) Let $A\in M_{3}(K)$
and let $X^3 = tX^2 - sX + d$ denote the characteristic polynomial of $A$.
Then, for all $n \geq 3$,
\begin{equation}\label{eq1}
A^n = a_{n-1}A + a_{n-2} Adj(A) + (a_n - ta_{n-1})\,I,
\end{equation}
where
\[
a_n =
\sum_{2i+3j \leq n} (-1)^i \binom{i+j}{ j} \binom{n-i-2j }{i+j} t^{n-2i-3j}s^id^j
\]
for $n>0$ and $a_0=1$.

(ii) Let $B \in M_2(K)$ and let $X^{2}=t\,X-d$ denote the characteristic polynomial of $B$.
Then, for all $n \geq 2$,
\[
B^n = b_n I + b_{n-1} Adj(B)
\]
for all $n \geq 2$, where
\[
b_n = \sum \binom{n-i }{ i} (-1)^i t^{n-2i} d^i.
\]
\end{corollary}

\begin{corollary}\label{c2}
Let $\theta \in K$,  $B \in M_2(K)$ and $t$ denote the trace and
$d$ the determinant of $B$. We have the following identity
in $M_2(K)$ :
\begin{equation*}
( a_{n-1} - \theta a_{n-2}) B + (a_n-(\theta+ t)a_{n-1}+ \theta a_{n-2} t)I
 =
y_{n-1}B + ( y_n - t\,y_{n-1}) I,
\end{equation*}
where
\begin{equation*}
a_n = \sum_{2i+3j \leq n} (-1)^i \binom{i+j }{ j} \binom{n-i-2j}{ i+j}
(\theta + t )^{n-2i-3j} (\theta t+ d)^i (\theta d )^j
\end{equation*}
and
\[
y_n = \sum \binom{n-i }{ i} (-1)^i t^{n-2i} d^i.
\]
In particular,
for any $\theta \in K$, one has
\[
b_n - ( \theta +1)b_{n-1} + \theta b_{n-2} = 1,
\]
where
\[
b_n = \sum_{2i+3j \leq n} (-1)^i \binom{i+j}{j} \binom{n-i-2j}{ i+j}
(\theta +2)^{n-2i-3j}(1+ 2 \theta)^i \theta^j.
\]
\end{corollary}
\begin{corollary}
The numbers $c_n = \sum_{2i+3j = n} (-1)^i \binom{i+j }{ j} 2^i3^j$ satisfy
\[
c_n+c_{n-1}-2c_{n-2} = 1.
\]
\end{corollary}
\begin{proof}
This is the special case of Corollary \ref{c2} where we take $\theta = -2$.
Note that the sum defining $c_n$ is over only those $i,j$ for which $2i+3j=n$.
\end{proof}

Note than when $k=3$, Theorem \ref{t1} can be rewritten as follows:
\begin{theorem}\label{t3}
Let $n$ be a positive integer and $x$, $y$, $z$ be indeterminates. Then
\begin{multline}\label{m1}
\sum_{2i+3j \leq n} (-1)^i \binom{i+j }{ j} \binom{n-i-2j }{ i+j} (x+y+z)^{n-2i-3j}
(xy+yz+zx)^i (xyz)^j\\
=
\frac{x\,y\,\left( x^{n+1} - y^{n+1} \right)  - x\,z\,\left( x^{n+1} - z^{n+1} \right)  +
    y\,z\,\left( y^{n+1} - z^{n+1} \right) }{\left( x - y \right) \,\left( x - z \right) \,
    \left( y - z \right) }.
\end{multline}
\end{theorem}
\begin{proof}
In Corollary \ref{c1}, let
\[
A=
\left (
 \begin{matrix}
x + y + z & 1 & 0 \cr
 - x\,y   - x\,z - y\,z & 0 & 1 \cr
 x\,y\,  z & 0 & 0 \cr
\end{matrix}
\right ).
\]
Then $t=x+y+z$, $s=xy+xz+yz$ and $d=xyz$. It is easy to show
(by first diagonalizing $A$) that the $(1,2)$ entry of $A^{n}$ equals the
right side of \eqref{m1}, with $n+1$ replaced by $n$, and the $(1,2)$ entry
on the  right side of \eqref{eq1} is $a_{n-1}$.
\end{proof}
\begin{corollary}
Let $x$ and $z$ be indeterminates and $n$ a positive integer. Then
\begin{multline*}
\sum_{2i+3j \leq n} (-1)^i \binom{i+j }{ j} \binom{n-i-2j}{ i+j}
 (2x+z)^{n-2i-3j}
(x^2+2xz)^i (x^2z)^j\\
=\frac{x^{2 + n} + n\,x^{1 + n}\,\left( x - z \right)  - 2\,x^{1 + n}\,z + z^{2 + n}}
  {{\left( x - z \right) }^2}.
\end{multline*}
\end{corollary}
\begin{proof}
Let $y\to x$ in Theorem \ref{t3}.
\end{proof}
Some interesting identities can be derived by specializing the
variables in Theorem \ref{t1}. For instance, in \cite{S04}, it was
noted that Binet's formula for the Fibonacci numbers is a
consequence of Theorem \ref{t1} for $k=2$. Here is a
generalization.

\begin{corollary}
(Generalization of Binet's formula)\\
Let the numbers $F_k(n)$ be defined by the recursion
\[F_k(0) = 1, F_k(r) = 0,\,\, \forall r < 0,\]
\[F_k(n) = F_k(n-1)+ F_k(n-2) + \cdots + F_k(n-k).\]
Then, we have
\[F_k(n) = \sum_{2i_2+ \cdots +ki_k \leq n} \frac{ (n-i_2-2i_3- \cdots -(k-1)i_k )!}{
i_1!i_2! \cdots i_k! (n-2i_2-3i_3- \cdots - ki_k )!}.\]
Further, this equals $\sum_{r_1+ \cdots + r_k=n} \lambda_1^{r_1} \cdots \lambda_k^{r_k}$
where $\lambda_i , 1 \leq i \leq k$ are the roots of the equation
$T^k-T^{k-1}-T^{k-2} - \cdots - 1 = 0$.
\end{corollary}
\begin{proof}
The recursion defining $F_k(n)$'s corresponds to the case $s_1=-s_2= \cdots = (-1)^{k-1}s_k=1$ of the theorem.
\end{proof}

\begin{corollary}
\begin{equation*}
\sum c(i_2,\cdots,i_k,n)
k^{n}\prod_{j=2}^{k} \left(
(-1)^{j-1}k^{-j}
\binom{k}{j}
\right)^{i_{j}}
=
\binom{n+k-1}{ k}.
\end{equation*}
where
\[
c(i_2,\cdots,i_k,n) =
\frac{ (n-i_2-2i_3- \cdots -(k-1)i_k )!}{
i_2! \cdots i_k! (n-2i_2-3i_3- \cdots - ki_k )!}.
\]
\end{corollary}
\begin{proof}
Take $x_i=1$ for all $i$ in Theorem 1. The left side of Theorem 1 is simply the sum
$\sum_{r_1 + \cdots + r_k = n} 1$.
\end{proof}
From Theorem\ref{t3} we have the following binomial identities as special cases.
\begin{proposition}
(i)
Let $\lambda$ be the unique positive real number satisfying $\lambda^3 = \lambda + 1$.
Let $x,y$ denote the complex conjugates such that
$xy = \lambda, x+y = \lambda^2,$ and let $z = -\frac{1}{\lambda}.$ Then,

\begin{multline*}
\sum_{2i+3j \leq n} (-1)^j \binom{n-2j}{j}
= \sum_{r+s+t=n} x^ry^sz^t\\
=
\frac{x\,y\,\left( x^{n+1} - y^{n+1} \right)  - x\,z\,\left( x^{n+1} - z^{n+1} \right)  +
    y\,z\,\left( y^{n+1} - z^{n+1} \right) }{\left( x - y \right) \,\left( x - z \right) \,
    \left( y - z \right) }.
\end{multline*}
(ii)
\[
\sum_{2i+3j \leq n} (-1)^j \binom{i+j}{j}\binom{n-i-2j}{i+j} = [(n+2)/2].
\]
(iii)
\[
\sum \binom{n-2j}{j} (-4)^j 3^{n-3j} = \frac{(3n+4)2^{n+1}+(-1)^n}{9}.
\]
(iv)
\begin{multline*}
\sum \binom{n-2j}{j} 3^{n-3j}(-2)^j\\ =
\frac{(1+\sqrt{3})^{n+1}-(1-\sqrt{3})^{n+1}}{2\sqrt{3}} +
\frac{(1+\sqrt{3})^{n+1}+(1-\sqrt{3})^{n+1}}{6} - \frac{1}{3}.
\end{multline*}
\end{proposition}

\section{Commutating Matrices}

In this section we derive various combinatorial identities by writing a general
$3 \times 3$ matrix $A$ as a product of commuting matrices.
\begin{proposition}\label{p2}
Let  $A$ be an arbitrary $3\times 3$ matrix with characteristic equation
$x^{3}-t x^{2}+s\,x-d=0$, $d \not = 0$. Suppose $p$ is arbitrary, with
$p^3+p^2t+ps+d \not = 0$, $p \not = 0,\,-t$. If $n$ is a positive integer, then
\begin{multline}\label{m2}
A^{n}=\left (\frac{p\,d}{p^3+p^{2}t+sp+d}\right )^{n}
\sum_{r=0}^{3n}
\sum_{j=0}^{n}
\sum_{k=0}^{n}
\binom{n }{ j}
\binom{n }{ k}
\binom{j }{ r-j-k}\\
\times
\left(\frac{-p(p+t)^{2}}{d}\right)^{j}
\left (\frac{-(p+t)}{p}\right)^{k}
 \left (\frac{-A}{p+t}\right)^{r}.
\end{multline}
\end{proposition}
\begin{proof}
This follows from the identity
\[
A=\frac{-1}{p^3+p^{2}t+sp+d}\left(p A^2  - A p (p +  t)- d\, I\right)\left(A + p\,I\right),
\]
after raising both sides to the $n$-th power and collecting powers of $A$.
Note that the two matrices $p A^2  - A p (p +  t)- d\, I$ and $A + p\,I$
commute.
\end{proof}

\begin{corollary}
Let $p$, $x$, $y$ and $z$ be indeterminates and let $n$
be a positive integer. Then
\begin{align*}
 \sum_{r = 0}^{3\,n}&
      \sum_{j = 0}^{n}\sum_{k = 0}^{n}
\binom{n }{ j}
\binom{n }{ k}
\binom{j }{ r-j-k}(-1)^{j+k+r}
\left(\frac{p(p+x+y+z)^{2}}{x y z}\right)^{j}\\
&\times
\left (\frac{p+x+y+z}{p}\right)^{k}
 \frac{x\,y\,\left( x^r - y^r \right)  -
  x\,z\,\left( x^r - z^r \right)  +
  y\,z\,\left( y^r - z^r \right)}{(p+x+y+z)^{r}}\\
\
&=\left(x\,y\,\left( x^n - y^n \right)  -
  x\,z\,\left( x^n - z^n \right)  +
  y\,z\,\left( y^n - z^n \right) \right)\\
&\phantom{adasaaasssasd}\times \left( \frac{p^3  +
         p^2\,\left( x + y + z \right)  +
         p\,\left( x\,y + x\,z + y\,z \right)+ x\,y\,z}
       {p\,x\,y\,z } \right)^n.
\end{align*}
\end{corollary}
\begin{proof} Let $A$ be the matrix from Theorem \ref{t3} and compare
$(1,1)$ entries on both sides of \eqref{m2}.
\end{proof}

\vspace{10pt}

\begin{corollary} Let $p$, $x$ and $z$ be indeterminates and let $n$
be a positive integer. Then
\begin{multline*}
 \sum_{r = 0}^{3\,n}
      \sum_{j = 0}^{n}\sum_{k = 0}^{n}
\binom{n}{  j}
\binom{n }{ k}
\binom{j }{ r-j-k}(-1)^{j+k+r}
\left(\frac{p(p+2x+z)^{2}}{x^2 z}\right)^{j}\\
\phantom{sdasdaasdasdada}\times
\left (\frac{p+2x+z}{p}\right)^{k}
 \frac{r\,x^{1 + r} - x^r\,z - r\,x^r\,z + z^{1 + r}}{(p+2x+z)^{r}}\\
%\end{multline*}
%\begin{multline*}
=\left(
n\,x^{1 + n} - x^n\,z - n\,x^n\,z + z^{1 + n}
\right)\\
\times
 \left( \frac{p^3  +
         p^2\,\left( 2x  + z \right)  +
         p\,\left( x^2 + 2x\,z  \right)+ x^{2}\,z}
       {p\,x^2\,z } \right)^n.
\end{multline*}
\end{corollary}
\begin{proof}
Divide both sides in the corollary
above by $x-y$ and let $y\to x$.
\end{proof}

\begin{corollary} Let $p$ and $x$  be indeterminates and let $n$
be a positive integer. Then
{\allowdisplaybreaks
\begin{multline*}
 \sum_{r = 0}^{3\,n}
      \sum_{j = 0}^{n}\sum_{k = 0}^{n}
\binom{n}{j}
\binom{n}{k}
\binom{j}{r-j-k}(-1)^{j+k+r}
\left(\frac{p(p+3x)^{2}}{x^3}\right)^{j}
\\
\times
\left (\frac{p+3x}{p}\right)^{k}
 \frac{r\,\left( 1 + r \right) \,x^{-1 + r}}{2(p+3x)^{r}}
=\frac{n\,\left( 1 + n \right) \,x^{-1 + n}}{2}
 \left( \frac{(p+x)^{3}}
       {p\,x^3 } \right)^n.
\end{multline*}}
\end{corollary}
\begin{proof}
Divide both sides in the corollary
above by $(x-z)^2$ and let $z\to x$.
\end{proof}

\begin{corollary}
Let $p$   be an  indeterminate and let $n$
be a positive integer. Then
\begin{multline*}
 \sum_{r = 0}^{3\,n}
      \sum_{j = 0}^{n}\sum_{k = 0}^{n}
\binom{n}{j}
\binom{n}{k}
\binom{j}{r-j-k}(-1)^{j+k+r}p^{j-k}(p+3)^{2j+k-r}
 \frac{r\,\left( 1 + r \right) }{2}\\
=\frac{n\,\left( 1 + n \right) }{2}
  \frac{(p+1)^{3n}}
       {p^{\,n}}.
\end{multline*}
\end{corollary}
\begin{proof}
Replace $p$ by $p\, x$ in the corollary above and simplify.
\end{proof}
Various combinatorial identities can be derived
from Theorem \ref{t3} by considering matrices $A$
 such that particular entries in $A^{n}$ have a simple
closed form.  We give four examples.

\vspace{10pt}

\begin{corollary}
 Let $n$ be a positive integer.\\
(i) If $p\not =0, -1$, then
\[
\sum_{r = 0}^{3\,n}
      \sum_{j = 0}^{n}\sum_{k = 0}^{n}
\binom{n}{j}
\binom{n}{k}
\binom{j}{r-j-k}(-1)^{j+k+r}p^{j-k}(p+3)^{2j+k-r}r=
n\frac{
{\left( 1 + p \right) }^{3n}
}{p^n}.
\]
(ii) Let $F_{n}$ denote the $n$-th Fibonacci number. If $p\not =0, -1, \phi$
or $1/\phi$ (where $\phi$ is the golden ratio, then
\begin{multline*}
\sum_{r = 0}^{3\,n}
      \sum_{j = 0}^{n}\sum_{k = 0}^{n}
\binom{n}{j}
\binom{n}{k}
\binom{j}{r-j-k}(-1)^{k+r}p^{j-k}(p+2)^{2j+k-r}F_{r}\\
=F_{n}\frac{(1+p)^{n}{\left( -1+p + p^{2} \right) }^{n}
}
{(-p)^n
}.
\end{multline*}
(iii) If $p\not =0, -1$ or $-2$, then
\begin{multline*}
\sum_{r = 0}^{3\,n}
      \sum_{j = 0}^{n}\sum_{k = 0}^{n}
\binom{n}{j}
\binom{n}{k}
\binom{j}{r-j-k}(-1)^{j+k+r}p^{j-k}(p+4)^{2j+k-r}2^{-j}(2^{r}-1)\\
=
(2^{n}-1)
\left(
\frac{
(1+p)^{2}(p+2)
}
{2p}
\right)^{n}.
\end{multline*}
(iv) If $p\not =0, -1, -g$ or $-h$ and $gh\not = 0$, then
\begin{multline*}
\sum_{r = 0}^{3\,n}
      \sum_{j = 0}^{n}\sum_{k = 0}^{n}
\binom{n}{j}
\binom{n}{k}
\binom{j}{r-j-k}(-1)^{j+k+r}p^{j-k}\\
\times(p+1+g+h)^{2j+k-r} \frac{g^{r}+h^{r}}{(g\,h)^{j}}
=
(g^{n}+h^n)
\left(
\frac{
 (1+p) (g+p) (h+p)
}
{g h p}
\right)^{n}.
\end{multline*}
\end{corollary}

\begin{proof} The results follow from considering the $(1,2)$ entries on both
sides in Theorem 3 for the matrices
\[
\left (
\begin{matrix} 1 & 1 & 0 \cr 0 & 1 & 0 \cr 0 & 0 & 1 \cr
\end{matrix}
\right ),
\left ( \begin{matrix} 1 & 1 & 0 \cr 1 & 0 & 0 \cr 0 & 0 & 1 \cr  \end{matrix}
\right ),
\left ( \begin{matrix} 3 & 1 & 0 \cr -2 & 0 & 0 \cr 0 & 0 & 1 \cr \end{matrix}
\right ),
\left ( \begin{matrix}
\displaystyle{\frac{g + h}{2} }& \displaystyle{\frac{{\left( g - h \right) }^2}{4}} & 0
 \cr 1 & \displaystyle{\frac{g + h}
   {2}} & 0 \cr 0 & 0 & 1 \cr \end{matrix}
\right),
\]
respectively.
\end{proof}

\section{A Result of Bernstein}
In \cite{B74} Bernstein showed that the only zeros of the integer function
\[
f(n):=\sum_{j\geq0} (-1)^{j}\binom{n-2j}{j}
\]
are at $n=3$ and $n=12$. We use Corollary \ref{c1} to relate
the zeros of this function to solutions of a certain cubic Thue equation
and hence to derive Bernstein's result.

Let
\[
A=
\left ( \begin{matrix} 1 & 1 & 0 \cr 0 & 0 & 1 \cr -1 & 0 & 0 \cr  \end{matrix}
\right ).
\]
With the notation of Corollary \ref{c1}, $t=1$, $s=0$, $d=-1$, so that
\[
a_n =
\sum_{3j \leq n} (-1)^j  \binom{n-2j }{j}=f(n),
\]
and, for $n \geq 4$,
\begin{align*}
A^{n}&= f(n-2)A^{2}+(f(n)-f(n-2))A+(f(n)-f(n-1))\,I\\
&=
\left (
\begin{matrix}
 f(n) & f(n-1) & f(n-2) \cr
-f(n-2) & -f(n-3) & -f(n-4) \cr
 -f(n-1) & -f(n-2) & -f(n-3) \cr
\end{matrix}
\right ).
\end{align*}
The last equality follows from the fact that $f(k+1)=f(k)-f(k-2)$,
for $k \geq 2$.

Now suppose $f(n-2)=0$. Since the recurrence relation above gives that
$f(n-4) =-f(n-1)$ and $f(n)=f(n-1)-f(n-3)$, it follows that
\begin{align*}
(-1)^{n}=
\det(A^{n})&=
\left |
 \begin{matrix}
 f(n-1)-f(n-3) & f(n-1) & 0 \cr
0& -f(n-3) & f(n-1) \cr
 -f(n-1) & 0& -f(n-3) \cr
\end{matrix}
\right |\\
&\phantom{as}\\
&=
-f(n-1)^{3}-f(n-3)^{3}+f(n-1)f(n-3)^{2}.
\end{align*}
Thus $(x,y) = \pm (f(n-1),f(n-3))$ is a solution of the Thue equation
\[
x^{3}+y^{3}-x\,y^{2}=1.
\]

One could solve this equation in the usual manner of finding bounds
on powers of fundamental units in the cubic number field defined by
the equation $x^{3}-x+1=0$. Alternatively, the Thue equation solver
in PARI/GP \cite{PG04} gives unconditionally (in less than a second) that the
only solutions to this equation are
\[
(x,y) \in \left \{ (4, -3), (-1, 1), (1, 0), (0, 1), (1, 1)\right \},
\]
leading to Bernstein's result once again.

 \allowdisplaybreaks{

}

\begin{thebibliography}{99}
\bibitem{B74}
Bernstein, Leon
\emph{Zeros of the functions
$f(n)=\sum_{i=0} (-1)^{i}\binom{n-2i}{i}$ }
J. Number Theory \textbf{6} (1974), 264--270.

\bibitem{McL04} James Mc Laughlin,
\emph{Combinatorial Identities Deriving from the $n$-th
Power of a $2 \times 2$ Matrix. }
Integers \textbf{4} (2004), A19, 14 pp.

\bibitem{PG04}
PARI/GP: \,\,\,  \href{http://pari.math.u-bordeaux.fr/}{http://pari.math.u-bordeaux.fr/}

\bibitem{S93} B. Sury,
\emph{A curious polynomial identity. }
Nieuw Arch. Wisk. (4) \textbf{11} (1993), no. 2, 93--96.

\bibitem{S04} -----,
\emph{A parent of Binet's forumla?}
Math. Magazine \textbf{77} (2004), no. 4, 308-310.



\end{thebibliography}
\end{document}